\documentclass{amsart}
\usepackage{amsfonts}
\usepackage{mathrsfs}
\usepackage{graphicx}
\usepackage[abs]{overpic}
\usepackage{geometry}
\usepackage{hyperref}
\usepackage{xypic}
\newtheorem{theorem}{Theorem}[section]
\newtheorem{lemma}[theorem]{Lemma}

\newtheorem{proposition}[theorem]{Proposition}

\theoremstyle{definition}

\newtheorem{question}[theorem]{Question}

\newcommand{\spinc}{\mathrm{Spin}^c}

\theoremstyle{remark}
\newtheorem{remark}[theorem]{Remark}

\numberwithin{equation}{section}

\begin{document}

\title{Hyperbolic 3-manifolds admitting no fillable contact structures}

\author{Youlin Li}
\address{Department of Mathematics \\Shanghai Jiao Tong University\\
Shanghai 200240, China}
\email{liyoulin@sjtu.edu.cn}

\author{Yajing Liu}
\address{Department of Mathematics \\UCLA\\
520 Portola Plaza, Los Angeles, CA 90095, USA
}
\email{yajingleo@math.ucla.edu}


\begin{abstract}
In this paper, we find infinite hyperbolic 3-manifolds that admit no weakly symplectically fillable contact structures,  using tools in Heegaard Floer theory. We also remark that part of these manifolds do admit tight contact structures. 
\end{abstract}

\maketitle

\section{Introduction}

There is a dichotomy for the contact structures on oriented 3-manifolds introduced by  Eliashberg: tight and overtwisted contact structures. The complete classification of overtwisted contact structures has been accomplished by Eliashberg in \cite{El1}. However, the fundamental questions on the  existence and classification of tight contact structures have still not been completely solved,  and they are widely open especially for hyperbolic 3-manifolds.  According to \cite{El2},  a weakly symplectically fillable contact structure is known to be tight.

\begin{question}[Question 8.2.2, \cite{El3} and Problem 4.142, \cite{Ki}]
\label{que:Ki}
Given an irreducible 3-manifold $M$, does it admit a tight or fillable contact structure?
\end{question}

There has not been a general method to answer Question \ref{que:Ki} for an arbitrary 3-manifold $M$. However, many important examples have been studied.  The first example of a 3-manifold that admits no weakly symplectically fillable contact structures has been found in \cite[Corollary 1.5]{Li1}. It is the Poincar\'{e} homology sphere with its natural orientation reversed.  In fact, it does not admit any tight contact structures either. See \cite{EtH1}. Subsequently, more examples of  3-manifolds that admit no weakly symplectically fillable contact structures have been constructed.  See \cite [Theorem 2.1]{Li2}, \cite [Theorem 4.2]{LiSt2},  \cite[Proposition 4.1]{LiSt3} and \cite[Corollary 3]{OwSt2}. All these examples are either Seifert fibered spaces or reducible 3-manifolds.  All of such Seifert fibered spaces are eventually classified in \cite[Theorem 1.5]{LeLi}. Furthermore, the existence problem of tight contact structures on Seifert fibered spaces has been solved in  \cite{LiSt4}.  Thus,  Question \ref{que:Ki} has been completely solved for all Seifert fibered spaces.

However, less examples of hyperbolic 3-manifolds have been studied. Basically, the techniques used in Seifert fibered spaces are hard to apply to hyperbolic manifolds.   Until the present paper, it wasn't known whether there is a hyperbolic 3-manifold which admits no fillable contact structures.   Tools in Heegaard Floer homology enable us to find such examples.  In this paper,  we give infinitely many hyperbolic manifolds which do not admit any weakly fillable contact structures.   

\begin{theorem}
\label{thm:main}
Suppose $q\geq4$ is a positive integer.  Let  $M_{q}$ be the $(2q+3)$-surgery on $S^{3}$ along the pretzel knot $P(-2,3,2q+1)$. If $2q+3$ is square free, then $M_{q}$ is a hyperbolic 3-manifold that admits no weakly symplectically fillable contact structures.
\end{theorem}

Since each prime number greater than $9$ is equal to some $2q+3$ for some integer $q$, we can find infinitely many $q$ such that $2q+3$ is square-free. Thus, we have obtained an infinite family of hyperbolic 3-manifolds which admit no weakly symplectically fillable contact structures.  While usually such proofs
resort to Donaldson's celebrated diagonalization theorem or Seiberg-Witten theory, our proof relies on a theorem of Owens-Strle on the correction terms and a theorem of Ozsv\'{a}th-Szab\'{o} on L-spaces in Heegaard Floer theory.

The 3-manifold $M_q$, where $q\geq4$ and $2q+3$ is square-free,  may also serve as a potential example of hyperbolic 3-manifold which admits no tight contact structures. We raise the following question.

\begin{question}
Suppose $q\geq4$, does $M_{q}$ admit a tight contact structure?
\end{question}


By computer experiments, we can find more examples of surgeries on $P(-2,3,2q+1)$ admitting no fillable contact structures, including rational surgeries on knots. Here, we give a proof for the following sequence of rational surgeries on $P(-2,3,7)$. 

\begin{theorem}
\label{thm:P(-2,3,7)}
Suppose $p\geq1$ is an integer. Let $M'_p$ be the $\frac{10p+1}{p}$-surgery on $S^{3}$ along the pretzel knot $P(-2,3,7)$.  If $10p+1$ is square-free, then $M'_p$  is a hyperbolic 3-manifold that admits no weakly symplectically fillable contact structures.
\end{theorem}

Note that there are infinitely many prime integers in form of $10p+1$, due to Dirichlet's theorem on arithmetic progressions. Thus, we have obtained another infinite family of hyperbolic 3-manifold that admit no weakly symplectically fillable contact structures.  By the same method, one can show that the $10$-surgery on $S^3$ along the pretzel knot $P(-2,3,7)$ does not admit any weakly symplectically fillable contact structures either.

The pretzel knot $P(-2,3,7)$ has the 4-ball genus $g_{s}=5$ and the maximal Thurston-Bennequin invariant $TB=9$. By \cite[Theorem 1.1]{LiSt2}, the $r$-surgery on $S^3$ along the pretzel knot $P(-2,3,7)$, $P(-2,3,7)_{r}$, admits a tight contact structure whenever $r>9$. By Theorem \ref{thm:P(-2,3,7)}, all tight contact structures on $M'_p$, where $10p+1$ is square-free, are not weakly symplectically fillable.  Also, Theorem \ref{thm:P(-2,3,7)} provides infinite examples of tight but not fillable contact structures. Admittedly, such examples of contact structures have been constructed on certain toroidal 3-manifolds and some hyperbolic 3-manifolds; see for example \cite{EtH2}, \cite{LiSt1}  and \cite{BEt}.  One should not be confused by the examples in \cite{BEt} with our examples, where it is not known whether their examples admit other fillable structures at all.  

\begin{remark}We compute Heegaard Floer homology with $\mathbb{F}=\mathbb{Z}/2\mathbb{Z}$ coefficients. Thus, by ``L-space", we mean ``$\mathbb{Z}/2\mathbb{Z}$-L-space". The results of Ozsv\'{a}th-Szab\'{o} on L-spaces in \cite{OzSz3} also apply to $\mathbb{Z}/2\mathbb{Z}$-L-spaces.
\end{remark}


\begin{remark}
After writing this paper, we are informed by John Etnyre that Amey Kaloti and B\"{u}lent Tosun have independently  found surgeries on $P(-2,3,2q+1)$ that admit no fillable contact structures \cite{KT}.
\end{remark}

\section*{Acknowledgement}
We are grateful to Ciprian Manolescu, Tye Lidman, and Zhongtao Wu for helpful discussions. We also want to thank John Etnyre and Ko Honda for their comments.  The first author was partially supported by grant no. 11471212 of the National Natural Science Foundation of China.

\section{Preliminaries}
To begin with, we show that $M'_p$, $p\geq1$, and  $M_{q}$, $q\geq4$, are hyperbolic.

\begin{lemma}
\label{lem:L_spaces}
Suppose $q\geq 4$ and $p\geq1$ are integers. Then $M_{q}$ and $M'_p$  are hyperbolic 3-manifolds.
\end{lemma}

\begin{proof}
Note that the pretzel knot $P(-2,3,2q+1)$ is a hyperbolic knot. Obviously, $M_{q}$ is neither $S^{3}$ nor $S^{1}\times S^{2}$. It is easy to know that the pretzel knot $P(-2,3,2q+1)$ is strongly invertible. So by \cite[Theorem 4]{Eu}, the cabling conjecture is true for $P(-2, 3, 2q+1)$. Hence $M_{q}$ is irreducible.  By \cite[Theorem 1.1]{Ma},  $P(-2,3,2q+1)$ admits no non-trivial cyclic surgery.  So $M_{q}$ is not a lens space.  By \cite[Theorem 1.1]{Me}, the small Seifert fibered surgeries on $P(-2,3,2q+1)$ are of slopes $4q+6$ and $4q+7$. So $M_{q}$ is not a small Seifert fibered space.  By \cite[Theorem 1.1]{Wu},  the toroidal surgeries on $P(-2,3,2q+1)$ are of slopes $4q+8$. So  $M_{q}$ is atoroidal. Therefore, according to Thurston's geometrization conjecture which was proved by Perelman, \cite{P1, P2, P3}, $M_{q}$ is hyperbolic.

The exceptional surgery coefficients of the pretzel knot $P(-2,3,7)$ are $16, 17, 18, 37/2, 19, 20$ and $1/0$. So $M'_{p}$ is hyperbolic by the same reason as above.
\end{proof}

According to \cite{OzSz4}, the pretzel knot $P(-2,3,2q+1)$ is an L-space knot for $q \geq 1$. Now we show that $M'_{p}$, $p\geq 1$, and  $M_q$, $q\geq 4$,  are actually L-spaces.  To this end, we recall the symmetrized Alexander polynomial of the pretzel knot $P(-2,3,2q+1)$, $q\geq 1$.

\begin{lemma}[Lemma 2.5, \cite{KL}]
\label{lem:AlexPoly}
The symmetrized Alexander polynomial of the pretzel knot $P(-2,3,2q+1)$ is $(-1)^{q-1}+\sum\limits_{j=1}^{q-1}(-1)^{q-1-j}(t^{j}+t^{-j})-(t^{q+1}+t^{-q-1})+(t^{q+2}+t^{-q-2})$, where $q\geq1$ is an integer.
\end{lemma}

For any rational number $r$, we denote the $r$-surgery on $K$ by $K_r$. According to \cite{OwSt1}, when $r$ is an integer, the $\spinc$ structures on $K_r$ can be indexed by integers $i$'s with $|i|\leq r/2$.

\begin{lemma}
Suppose $q\geq 4$ and $p\geq1$ are integers. Then $M_{q}$ and $M'_p$  are L-spaces.
\end{lemma}

\begin{proof}
For any L-space knot, by \cite[Theorem 1.2]{OzSz3} and \cite[Theorem 1.2]{OzSz4}, the Seifert genus is the power of the leading term in the symmetrized Alexander polynomial. So by Lemma~\ref{lem:AlexPoly}, the Seifert genus of $P(-2,3,2q+1)$ is $q+2$, for $q\geq1$.  According to \cite[Proposition 9.6]{OzSz6}, $M'_p$, $p\geq1$, and $M_{q}$, $q\geq4$, are L-spaces.  In fact, for any L-space knot $K$, the surgery manifold $K_{r}$ with $r\in \mathbb{Q}$ is an L-space for all $r\geq 2g(K)-1$, where $g(K)$ is the Seifert genus of $K$.
\end{proof}

We will show that any of  $M_q$ with $q\geq4$ and $M'_p$ with $p\geq1$  cannot bound negative-definite 4-manifolds. We utilize the following theorem of Owens and Strle.

\begin{theorem}[Theorem 2, \cite{OwSt1}]
\label{thm:nd}
Let $Y$ be a rational homology sphere with $|H_1(Y;\mathbb{Z})|=\delta$. If $Y$ bounds a negative-definite 4-manifold $X$, and if either $\delta$ is square-free or there is no torsion in $H_1(X;\mathbb{Z})$, then
\[
\max_{\mathfrak{t}\in \spinc(Y)} 4d(Y, \mathfrak{t})\geq
\begin{cases}
1-1/\delta &\mbox{if $\delta$ is odd,}\\
1                                &\mbox{if $\delta$ is even.}
\end{cases}
\]
\end{theorem}

Suppose $K$ is an L-space knot.  Then the d-invariants of integral surgeries on $K$ can be computed by the following formula.

\begin{theorem}[Therem 6.1, \cite{OwSt1}]
\label{eq:d-invt_L_space_knot}
Given $n\in \mathbb{N}$ and  $\forall | i |\leq n/2,$ we have
\begin{align*}
                       d(K_{n},i)&=d(U_{n},i)-2t_{i}(K)\\
&=\frac{(n-2|i|)^2}{4n}-\frac{1}{4}-2t_i(K),
\end{align*}
where $U$ is the unknot,  the torsion coefficient \[t_i(K)=\sum_{j>0}j a_{|i|+j},\] and $a_{k}$ is the coefficient of $t^k$ in the normalized symmetrized Alexander polynomial $\Delta_{K}(t)$ of $K$.
\end{theorem}

For positive rational surgeries on a knot $K$, the d-invariants can be computed by the following formula given by Ni and Wu.
\begin{proposition}[Proposition 1.6, \cite{NW}]
Suppose $g,h$ are positive coprime integers and fix $0\leq i\leq g-1.$ Then,
\[d(K_{g/h},i)=d(L(g,h),i)-2 \max\{V_{\lfloor\frac{i}{h}\rfloor},H_{\lfloor\frac{i-g}{h}\rfloor}\}.\]
\end{proposition}

Here, $V_s$ and $H_s$ with $s \in \mathbb{Z}$ are knot invariants coming from knot Floer complex; see \cite{NW}. In fact, $H_{s}=V_{-s}$. Thus, we have the following formula,
\begin{equation}
\label{eq:d-invt_rational}
d(K_{g/h},i)=d(L(g,h),i)-2 \max\{V_{\lfloor\frac{i}{h}\rfloor},V_{-\lfloor\frac{i-g}{h}\rfloor}\}.
\end{equation}

\section{Integral surgery on $P(-2, 3, 2q+1)$}

We compute the torsion coefficients and the d-invariants of $M_q$ as follows.

\begin{lemma}
\label{lem:t_j}
For any integer $k\geq 1$, the pretzel knot $K=P(-2,3,4k+3)$ has the following formulas for the torsion coefficients $t_j(K)$:
\[
t_j(K)=\begin{cases} k+1-\lfloor\frac{j}{2}\rfloor &\mbox{$j=0,\ldots,2k+1,$}\\
                     1 &\mbox{$j=2k+2,$}\\
                     0 &\mbox{$j\geq 2k+3.$}
\end{cases}
\]

For any integer $k\geq 1$, the pretzel knot $K=P(-2,3,4k+1)$ has the following formulas for the torsion coefficients $t_j(K)$:
\[
t_j(K)=\begin{cases} k+1-\lfloor\frac{j+1}{2}\rfloor &\mbox{$j=0,\ldots,2k,$}\\
                     1 &\mbox{$j=2k+1,$}\\
                     0 &\mbox{$j\geq 2k+2.$}
\end{cases}
\]

\end{lemma}

\begin{proof}Suppose $K=P(-2,3,4k+3)$.
By definition, we have the recursion relation for $t_j(K)$,
$$t_{j-1}(K)=t_j(K)+\sum_{i\geq j}a_i(K), \quad \forall j\geq 1.$$
 By Lemma \ref{lem:AlexPoly}, we have that
\[\sum_{i\geq j}a_i(K)=\begin{cases}1 &\mbox{ if $j=2k+3,$ or  $j=2i$ for all $0\leq i\leq k,$}\\
                                 0 &\mbox{ if $j=2k+2,$ or  $j=2i+1$ for all $0\leq i \leq k.$}
                                 \end{cases}\]
It is also easy to see $t_j(K)=0$ for all $j\geq 2k+3$ and $t_{2k+2}(K)=1.$ Combing this with the recursion formula, we can prove the result for $K=P(-2,3,4k+3)$.

Similarly, we can prove the results for $K=P(-2,3,4k+1)$.
\end{proof}

\begin{lemma}
\label{lem:d_jq}
For all integers $q\geq 4$, all the d-invariants on $M_{q}$ are negative.
\end{lemma}

\begin{proof}
Suppose $q=2k+1$. By Lemma \ref{lem:t_j} and Theorem \ref{eq:d-invt_L_space_knot}, for all $0\leq j\leq 2k+1,$
\begin{align*}
d(K_{4k+5},\pm j)&=\frac{(4k+5-2j)^2}{4(4k+5)}-\frac{1}{4}-2(k+1-\lfloor\frac{j}{2}\rfloor)\\
             &\leq \frac{4k+5}{4}-j+\frac{j^2}{4k+5}-\frac{1}{4}-2(k+1-\frac{j}{2})\\
             &=-k-1+\frac{j^2}{4k+5}\\
             &\leq -k-1+\frac{(2k+1)^2}{4k+5}\\
             &<0.
\end{align*}
Furthermore,  $$d(K_{4k+5},\pm(2k+2))=\frac{1}{4(4k+5)}-\frac{1}{4}-2<0.$$
Thus, all the $d$-invariants are negative for the surgery $K_{4k+5}$ when $K=P(-2,3,4k+3)$.

The case when $q=2k$ is similar.
\end{proof}



According to Theorem \ref{thm:nd} and Lemma \ref{lem:d_jq}, we have the following lemma.

\begin{lemma}
\label{lem:bnd}
Suppose $q\geq4$ and $2q+3$ is square free,  then $M_{q}$ cannot bound any negative-definite 4-manifold.
\end{lemma}

Now we turn back to the proof of the Theorem \ref{thm:main}.
\begin{proof}[Proof of Theorem \ref{thm:main}]
We will prove a slightly stronger result that $M_{q}$ does not admit any weak symplectic semi-fillings.
Suppose on the contrary that $M_{q}$ admits a contact structure which has a weak symplectic semi-filling $X$.  Since $M_{q}$ is an L-space,  by \cite[Theorem 1.4]{OzSz3},   $X$  must have only one boundary component and have $b^{+}_{2}(X)=0$.  Since $M_{q}$ is a rational homology sphere, the intersection form of $X$ is non-degenerate, and thus negative-definite.  This contradicts with Lemma \ref{lem:bnd}.  We finish the proof of Theorem \ref{thm:main}.
\end{proof}

\section{Rational surgeries on $P(-2,3,7)$}
The definition of the invariants $V_s$'s from Equation \eqref{eq:d-invt_rational} comes from the complex $A^+_s$ that appears in the surgery formulas in Heegaard Floer homology. In the case of L-space knots we can compute them using the Alexander polynomials.
\begin{lemma}
\label{lem:computing_V_s}
Suppose $K$ is an L-space knot and its symmetrized Alexander polynomial is $\Delta_{K}(t).$ We rewrite it in the form
\[\frac{t}{t-1}\cdot \Delta_{K}(t)=\sum_{i}a_i \cdot t^i.\]
Then, we have the following formulas for computing $V_s, \forall s\in \mathbb{Z},$
\[V_s=\sum_{i\geq s+1}a_i.\]
\end{lemma}

\begin{proof}
Let $C$ denote the knot Floer complex $CFK^{\infty}(K)$ induced by a doubly-pointed Heegaard diagram of $K$.
We have the following commutative diagram where the two rows are short exact sequence of complexes,
\[
\xymatrix{0\ar[r] &C\{\max(i,j-s)<0\} \ar[r]\ar[d] & C\ar[r]\ar[d] & C\{\max(i,j-s)\geq 0\}\ar[r] \ar[d]&0\\
          0\ar[r] &C\{i<0\} \ar[r]          & C \ar[r]& C\{i\geq 0\}\ar[r] & 0.
}
\]

Since $A^-_s=C\{\max(i,j-s)\leq0\}\cong C\{\max(i,j-s)<0\}$ and $A^-_{+\infty}$ is isomorphic to $C\{i\leq 0\}\cong C\{i<0\}$, the above diagram translates into
\[
\xymatrix{0\ar[r] &A^-_s \ar[r]\ar[d]_{I_s} & C\ar[r]\ar[d] & A^+_s\ar[r] \ar[d]^{v_s}&0\\
          0\ar[r] & A^-_{+\infty} \ar[r]          & C \ar[r]& B^+_{s}\ar[r] & 0.
}
\]

Since $K$ is an L-space knot, we have that
$A^-_s=\{{\bf{x}}\in CFK^{\infty}(K)|A({\bf{x}})\leq s\}$ has  homology $\mathbb{F}[U]$.  By passing to the long exact sequence of homologies, one can show that the induced map on homology $(I_s)_*:\mathbb{F}[U]\to\mathbb{F}[U]$ is a multiplication of some power $U^{n_s}$ and 
$V_s=n_s.$

From the long exact sequence induced by the short exact sequence  $0\to A^-_{i-1}\xrightarrow{\iota_{i-1}} A^-_{i}\to A^-_i/A^-_{i-1}\to 0,$ we have the following exact sequence $0\to \mathrm{coker}((\iota_{i-1})_*)\to H_*(A^-_i/A^-_{i-1})\to \ker((\iota_{i-1})_*)\to 0.$ From $I_{i-1}=I_i\circ \iota_{i-1}$, it follows that $\iota_{i-1}$ cannot be 0 acting on homology. Thus, $\ker((\iota_{i-1})_*)=0$, and
$H_*(A^-_{i}/A^-_{i-1})=\mathrm{coker}((\iota_{i-1})_*)$.
Hence, we have $n_{i-1}-n_{i}=\mbox{\Large$\chi$}(A^-_{i}/A^-_{i-1}).$

Since $CFK^-(K,i)=A^-_{i}/A^-_{i-1}$, by Proposition 9.2 in \cite{OzSz5}, we have that
\begin{align*} n_s&=\sum_{i\geq s+1}\mbox{\Large$\chi$}(HFK^-(K,i))\\
                  &=\sum_{i\geq s+1}a_i.
\end{align*}
\end{proof}

\begin{lemma}
Let $K$ be the $P(-2,3,7)$ pretzel knot. For all $p\geq 1$, all the d-invariants on $K_{\frac{10p+1}{p}}$ are non-positive.
\end{lemma}

\begin{proof}
By Equation \eqref{eq:d-invt_rational}, we have $\forall 0\leq i\leq 10p,$
\begin{align*}
d(K_{\frac{10p+1}{p}},i)&=d(L(10p+1,p),i)-2\max\{V_{\lfloor\frac{i}{p}\rfloor}, V_{-\lfloor\frac{i-10p-1}{p}\rfloor}\}.
\end{align*}

By Lemma \ref{lem:computing_V_s} and Lemma \ref{lem:AlexPoly}, we know $$V_0=V_1=2,\quad V_2=V_3=V_4=1,\quad V_p=0,\  \forall p\geq 5.$$

By Proposition 4.8 in \cite{OzSz1}, we have
\begin{align*}
d(L(10p+1,p),i)&=-\frac{1}{4}+\frac{(11p-2i)^2}{4p(10p+1)}-d(L(p,1),j)\\
&=\frac{(11p-2i)^2}{4p(10p+1)}-\frac{(p-2j)^2}{4p},
\end{align*}
with $0\leq j\leq p-1$ being the residue of $i$ modulo $p$.

Now we will discuss the correction terms case by case.
\begin{enumerate}
\item If $0 \leq i\leq 2p-1$, then $\lfloor\frac{i}{p}\rfloor=1$. Then,  $V_{\lfloor\frac{i}{p}\rfloor}=2.$ Thus, $$\max\{V_{\lfloor\frac{i}{p}\rfloor}, V_{-\lfloor\frac{i-10p-1}{p}\rfloor}\}=2.$$

    Therefore,
    \[d(K_{\frac{10p+1}{p}},i)\leq \frac{(11p)^2}{4p \cdot10p}-2\cdot2<0.\]
\item If $2p \leq i\leq 5p-1$, then  $\lfloor\frac{i}{p}\rfloor=2,3,4$. Then, $V_{\lfloor\frac{i}{p}\rfloor}=1$. So $$\max\{V_{\lfloor\frac{i}{p}\rfloor}, V_{-\lfloor\frac{i-10p-1}{p}\rfloor}\}=1.$$

   In addition, since $2p \leq i\leq 5p-1$, we have $(11p-2i)^2\leq (7p)^2.$
    Thus,
    \[d(K_{\frac{10p+1}{p}},i)\leq \frac{(7p)^2}{4p\cdot 10p}-2\cdot1<0.\]
\item If $5p \leq i\leq 6p-1$, then  let
\[i=5p+j,\]
where $0\leq j\leq p-1.$

    Thus,
    \begin{align*}
    d(K_{\frac{10p+1}{p}},i)&\leq \frac{(11p-2i)^2}{4p\cdot (10p+1)}-\frac{(p-2j)^2}{4p}\\
    &=\frac{(p-2j)^2}{4p\cdot (10p+1)}-\frac{(p-2j)^2}{4p}\\
    &\leq 0.
    \end{align*}

\item If $i=6p,$ then
\[d(K_{\frac{10p+1}{p}},i)\leq \frac{p^2}{4p\cdot (10p+1)}-\frac{p^2}{4p}<0.\]
\item If $6p+1\leq i\leq 8p+1$, then  $-\lfloor\frac{i-10p-1}{p}\rfloor=2,3,4$. Then, $V_{-\lfloor\frac{i-10p-1}{p}\rfloor}=1.$ So $$\max\{V_{\lfloor\frac{i}{p}\rfloor}, V_{-\lfloor\frac{i-10p-1}{p}\rfloor}\}=1.$$

   In addition, since $6p+1 \leq i\leq 8p+1$, we have $(11p-2i)^2\leq (7p)^2.$
    Thus,
    \[d(K_{\frac{10p+1}{p}},i)\leq \frac{(7p)^2}{4p\cdot 10p}-2\cdot1<0.\]
\item If $8p+2\leq i\leq 10p$, then either $-\lfloor\frac{i-10p-1}{p}\rfloor=1$. Then,  $V_{-\lfloor\frac{i-10p-1}{p}\rfloor}=2.$ Thus, $$\max\{V_{\lfloor\frac{i}{p}\rfloor}, V_{-\lfloor\frac{i-10p-1}{p}\rfloor}\}=2.$$

    Therefore,
    \[d(K_{\frac{10p+1}{p}},i)\leq \frac{(11p)^2}{4p \cdot10p}-2\cdot2<0.\]
\end{enumerate}
\end{proof}

Now we are ready to prove Theorem \ref{thm:P(-2,3,7)}.
\begin{proof}[Proof of Theorem \ref{thm:P(-2,3,7)}]
We follow the same line as in Theorem \ref{thm:main}. By Lemma \ref{lem:L_spaces}, we have $M'_p$ are all L-spaces.  As long as $10p+1$ is square-free,  by Theorem \ref{thm:nd}, $M'_p$ does not bound any negative definite 4-manifolds and thus does not admit any weakly symplectically fillable contact structures.  We finish the proof.
\end{proof}

\end{document}